\documentclass[smallextended]{svjour3}
\usepackage{amsmath}
\usepackage{amsmath}
\usepackage{amssymb}
\usepackage{comment}
\usepackage{graphicx}

\newcommand{\qo}{\mathcal{Q}}
\newcommand{\po}{\mathcal{P}}

\begin{document}

\title{A finite chiral 4-polytope in $\mathbb{R}^4$\thanks{Partially supported by PAPIIT - UNAM under the project IN112511 and by CONACyT under project 166951}}

\author{Javier Bracho \and Isabel Hubard \and Daniel Pellicer}
\institute{ J. Bracho \and I. Hubard \at Instituto de Matem\'{a}ticas, Universidad Nacional Aut\'{o}noma de M\'{e}xico, M\'{e}xico \ \ \email{jbracho, isahubard@im.unam.mx}
\and D. Pellicer \at
Centro de Ciencias Matem\'{a}ticas, Universidad Nacional Aut\'{o}noma de M\'{e}xico, M\'{e}xico \ \ \email{pellicer@matmor.unam.mx}}

\date{\today}

\maketitle

\begin{abstract}
In this paper, we give an example of a chiral 4-polytope in projective 3-space.
This example naturally yields a finite chiral 4-polytope in Euclidean 4-space, giving a counterexample to Theorem 11.2 of \cite{peter}.
\end{abstract}


\section{Introduction}

Abstract polytopes are combinatorial structures that resemble convex polytopes.
Of particular interest are those with high degree of symmetry, together with their realisations in euclidean spaces.
Regular polytopes have maximum degree of symmetry, with their automorphism group
being as bis as possible.
Chiral polytopes have maximum possible rotational symmetry, but no reflections.  (See \cite{arp} for formal definitions of these concepts.)

A finite (abstract) $n$-polytope is said to be of full rank if it can be realised in euclidean $n$-space.
In 1977 Gr\"unbaum (\cite{branko}) gave the list of the finite regular 18 polyhedra of full rank and a few years later Dress (\cite{dress1,dress2}) showed that the list was complete.
Regular polytopes of full rank have been studied by McMullen in \cite{peter}.
Theorem 11.2 of the same paper claims that there are no chiral $n$-polytopes of full rank.
In \cite{egon1,egon2} Schulte independently proved this result for $n=3$.
Here paper we give a counterexample to the theorem, for $n=4$.

Our approach follows \cite{proj1,proj2}, where a vertex-transitive realisation of a finite polytope in $\mathbb{R}^{n+1}$ naturally corresponds to a projective polytope in $\mathbb{P}^n$.
Hence, to give a chiral $4$-polytope of full rank we construct a chiral 4-polytope in the projective space $\mathbb{P}^3$.

\section{A chiral 4-polytope in $\mathbb{P}^3$}

In this section we construct a chiral 4-polytope of Schl\"afli type $\{4,3,3\}$ \cite[pp. 29, 30]{arp}. To this end, we start by considering the complete bipartite graph $K_{4,4}$. As it is shown in Figure~\ref{K44}, we can colour the edges of $K_{4,4}$ with 4 colours in such a way that two edges of the same colour are not incident, so that
each colour induces a perfect matching in the graph.
We label the vertices of the graph $v_0, v_1, v_2, v_3, u_0, u_1, u_2, u_3$ as in the figure.

\begin{figure}[htbp]
\begin{center}
\includegraphics[width=7cm]{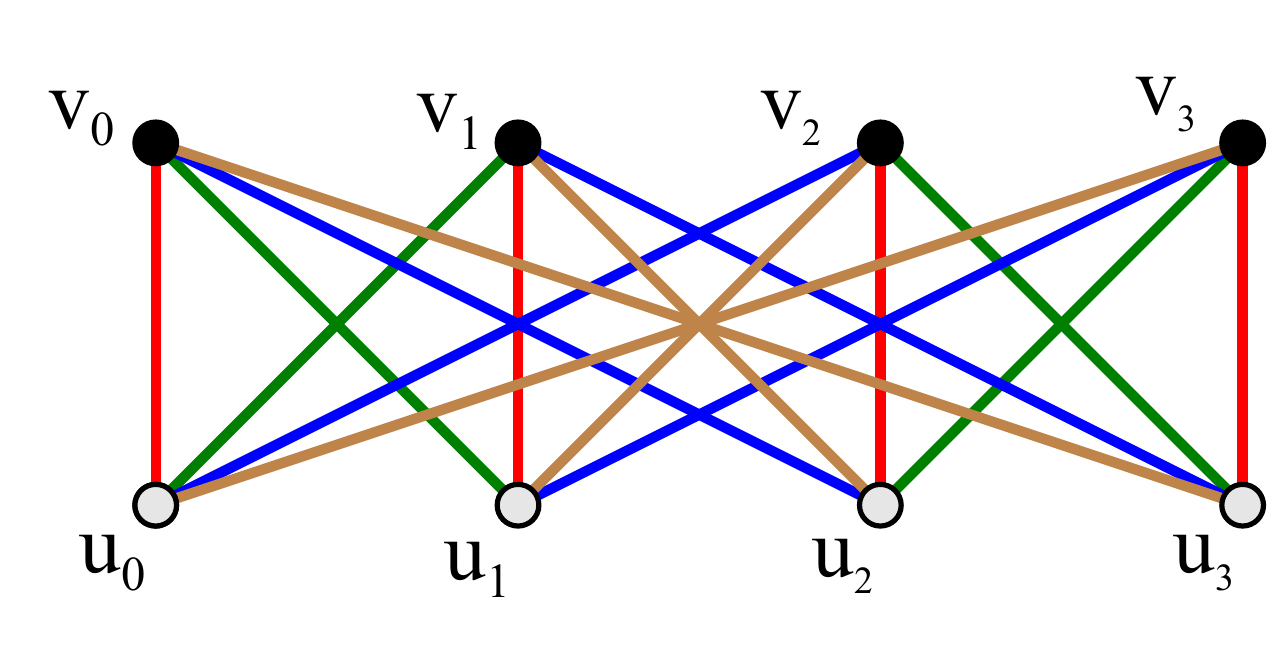}
\caption{The graph $K_{4,4}$, with the edges coloured with 4 colours.}
\label{K44}
\end{center}
\end{figure}

It is not difficult to see that with the colouring of $K_{4,4}$ given in Figure~\ref{K44}, we can obtain a colourful 4-polytope $\po$ in the sense of \cite{colour}.
In fact,  the 2-faces of $\po$ are the cycles of $K_{4,4}$ that have exactly two colours.
Hence, each of the alternating squares of two given colours is a 2-face of $\po$.
The facets of $\po$ are defined by the subgraphs coloured with exactly tree colours.
Then, we can see that $\po$ has 4 facets and each of them is a cube.
(In fact, we observe that the graph $Q_3$ of the cube is precisely $K_{4,4}$ minus a perfect matching.)
The automorphisms of $\po$ are all the colour respecting automorphisms of $K_{4,4}$, that is, all the automorphisms of $K_{4,4}$ that induce a permutation on the colours (see \cite{colour}).

\begin{figure}[htbp]
\begin{center}
\includegraphics[width=4.5cm]{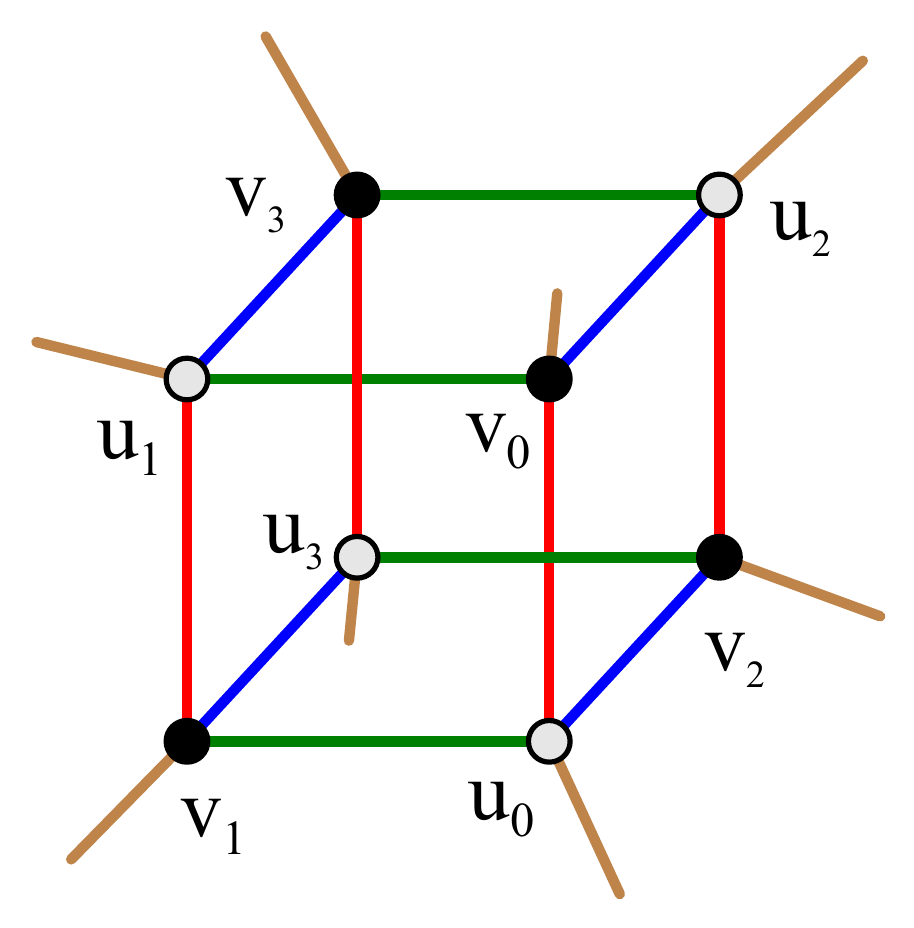}
\caption{The  {\em regular colouring} of $K_{4,4}$ in $\mathbb{P}^3$.}
\label{hemicube}
\end{center}
\end{figure}
Therefore, $\po$ is isomorphic to the hemi-hypercube $\{4,3,3\}/2$ shown in Figure~\ref{hemicube}, and hence we can think that $\po$ lives in the projective 3-space.
(Recall that the hemi-hypercube in $\mathbb{P}^3$ can be understood as the quotient of the hypercube $\{4,3,3\}$ in $\mathbb{R}^4$ by the central inversion of $\mathbb{R}^4$ that identifies antipodal points of the hypercube. In fact the vertices of  $\{4,3,3\}/2$, in homogeneous coordinates, have all entries $\pm 1$, and the edges are the geodesics between two vertices that differ in one entry. Furthermore, the four colours of our embedding of $K_{4,4}$ correspond to the four coordinates, or directions.)
%
Observe that there is an edge between opposite vertices of a given facet of $\{4,3,3\}/2$.
In fact, the edge has precisely the colour that is missing in that cube. (See Figure~\ref{hemicube}). 
It is well-know that the hemi-hypercube $\{4,3,3\}/2$ is a regular 4-polytope in the projective space.

%
It is well-know that the hemi-hypercube $\{4,3,3\}/2$ is a regular 4-polytope in the projective space.
Note that symmetries of the hemi-hypercube correspond not only to the colour respecting automorphisms of the graph, but also to the isometries of $\mathbb{P}^3$ that preserve the graph.

Now, consider the graph $K_{4,4}$ embedded in  $\mathbb{P}^3$ as fixed, and observe that it admits two \emph{chiral colourings} as in Figure~\ref{cons3}.
They are combinatorially equivalent to the colouring of Figures~\ref{K44} and~\ref{hemicube}, because all its bi-coloured cycles are squares.
(In fact, they correspond to the Petrie polygons of the hemi-hypercube).

%
%
%
\begin{figure}[htbp]
\begin{center}
\includegraphics[width=11cm]{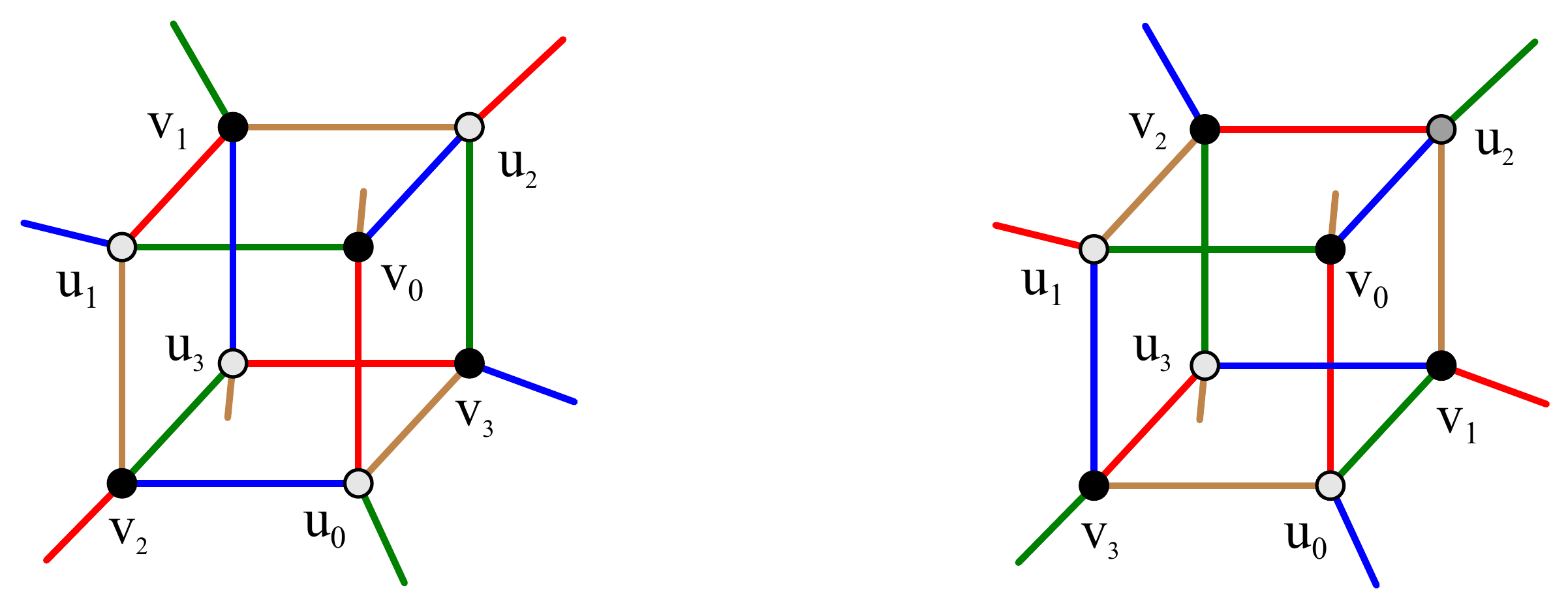}
\caption{The two chiral colourings of of $K_{4,4}$ in $\mathbb{P}^3$.}
\label{cons3}
\end{center}
\end{figure}

A simple inspection shows that the two chiral colourings have the following properties: (a) each colour has an edge in each direction (is transversal to the regular colouring), and (b) each 2-face of the regular hemi-hypercube has the four colours.
It is not hard to see that any of these properties defines the chiral colourings

We now analyse these new 4-polytopes. It should be clear that what we say about one of them can be similarly said about the other, and hence we use the one on the left of Figure~\ref{cons3}.

As before, we regard this new 4-polytope $\qo$ as a colourful polytope: the 2-faces of $\qo$ are the cycles of exactly two colours and the facets are determined by the subgraphs with exactly 3 colours.
Hence, we see that the 2-faces are again 4-gons, that we see now as helices in the projective space (see Figure~\ref{cons4}).
In fact, each of the 2-faces of $\qo$ corresponds to a Petrie polygon of $\po$ (see \cite[p. 163]{arp}).

\begin{figure}[htbp]
\begin{center}
\includegraphics[width=11cm]{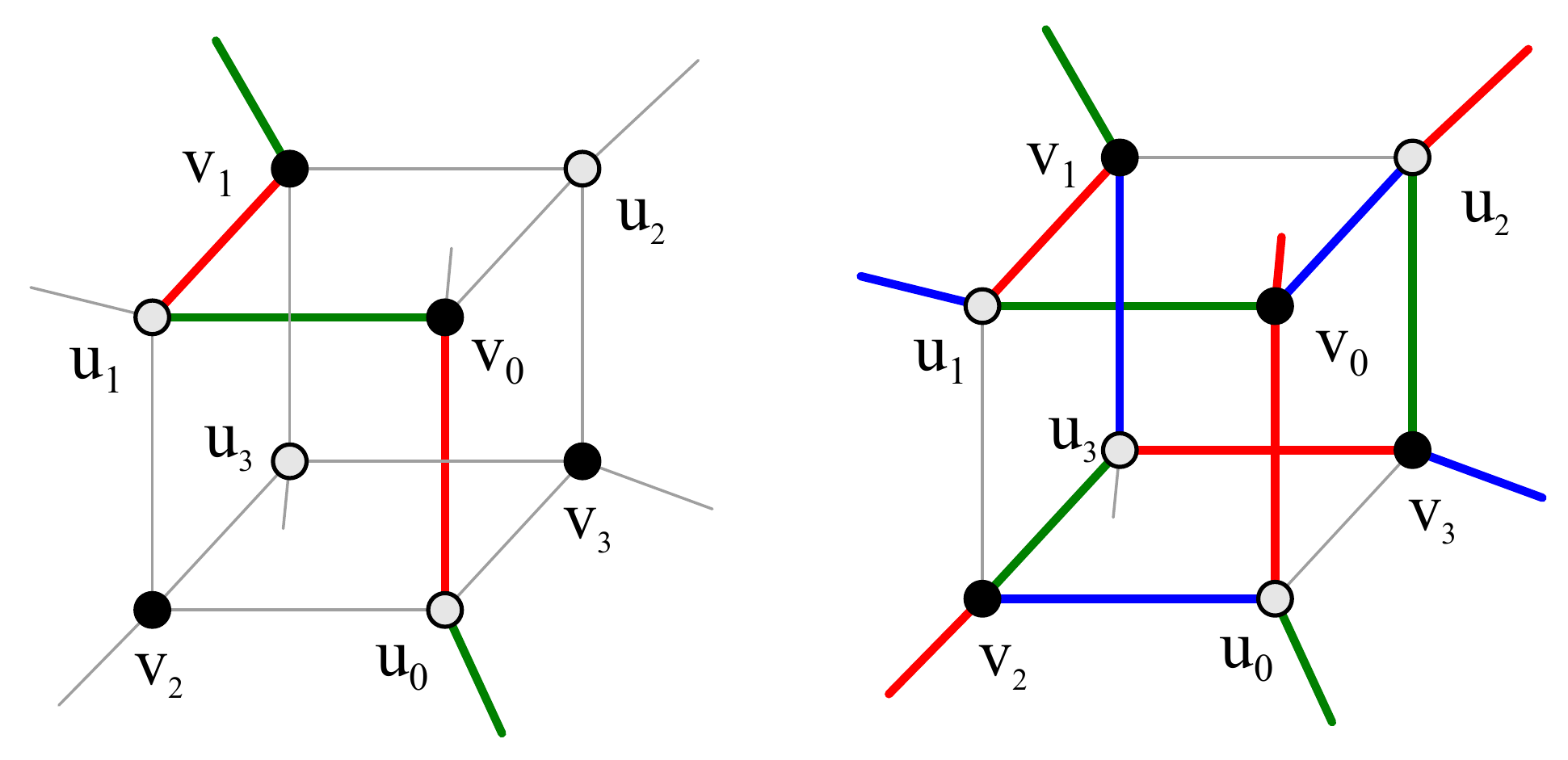}
\caption{A 2-faces and a 3-face of $\qo$.}
\label{cons4}
\end{center}
\end{figure}
%

We note that the edge colouring of $K_{4,4}$ in $\qo$ is the same as its colouring in $\po$. Hence $\po$ and $\qo$ are combinatorially isomorphic. On the other hand, the 2-faces of $\qo$ are Petrie polygons of the hemicube $\po$ and vice-versa.

\begin{theorem}
Let $\qo$ be the 4-polytope of type $\{4,3,3\}$ in the projective space constructed above. Then $\qo$ is combinatorially isomorphic to the hemi-hypercube $\{4,3,3\}/2$ and geometrically chiral, with geometrically chiral facets.
\end{theorem}

\begin{proof}
The 1-skeleton of $\qo$ is exactly the same as the 1-skeleton of the hemi-hypercube $\po$ and every isometry of $\mathbb{P}^3$ that preserves such graph is a symmetry of $\po$, then every symmetry of $\qo$ is a symmetry of $\po$. Recall that $\po$ has 192 symmetries.

We first note that all 4-fold rotations of the central cube in the left part of Figure \ref{cons3} preserve the coloration by cyclically permuting the four colours. It is also true that the rotation around an edge is a 3-fold rotation that preserves the colouring by cyclically permuting the three other colours. In particular, the symmetry group of $\qo$ contains the rotation subgroup of the central cube and acts transitively on the cubes of $\po$.
Consequently, all 96 orientation preserving elements of the symmetry group of $\po$ belong also to the symmetry group of $\qo$. Observe that the reflection with respect to the plane perpendicular to the four edges with a given color in the regular colouring (Figure~\ref{hemicube}) does not preserve colors in the chiral colouring and hence, it is not a symmetry of $\qo$. This implies that $\qo$ has precisely $96$ symmetries.


It only remains to verify that these $96$ elements are precisely the combinatorial rotations of the colourful polytope $\qo$. We observe that the stabilisers of a pair consisting of incident 2-face and 3-face are generated by the twist along an axis of the (helical) 2-face (note that this twist preserves the coloring). For example, the stabiliser of the 2-face in Figure~\ref{cons4}, as permutation of the vertices of $\qo$, is  $(v_0u_0v_1u_1)(u_2v_2u_3v_3)$. Furthermore, the stabiliser of a pair consisting of incident vertex and 3-face is generated by the 3-fold rotation around the edge of $\qo$ containing the vertex, but not contained on the 3-face. Finally, the edge pointwise stabiliser is generated by the 3-fold rotation around that edge. Following \cite[p. 495]{chiregas}, this implies that $\qo$ is indeed chiral.

The arguments on the previous two paragraphs also imply that the facets are geometrically chiral.
\end{proof}

We end this section by pointing out that the two enantiomorphic forms of $\qo$ are indeed the two polytopes arising from the diagrams in Figure\ref{cons3}.

\section{Finite chiral 4-polytopes in $\mathbb{R}^4$}

In the previous section we gave an example of a chiral 4-polytope $\qo$ in $\mathbb{P}^3$.
We now take the double cover $\hat{\qo}$ of $\qo$, in the sphere $\mathbb{S}^3 \subset \mathbb{R}^4$.
\begin{figure}[htbp]
\begin{center}
\includegraphics[width=5cm]{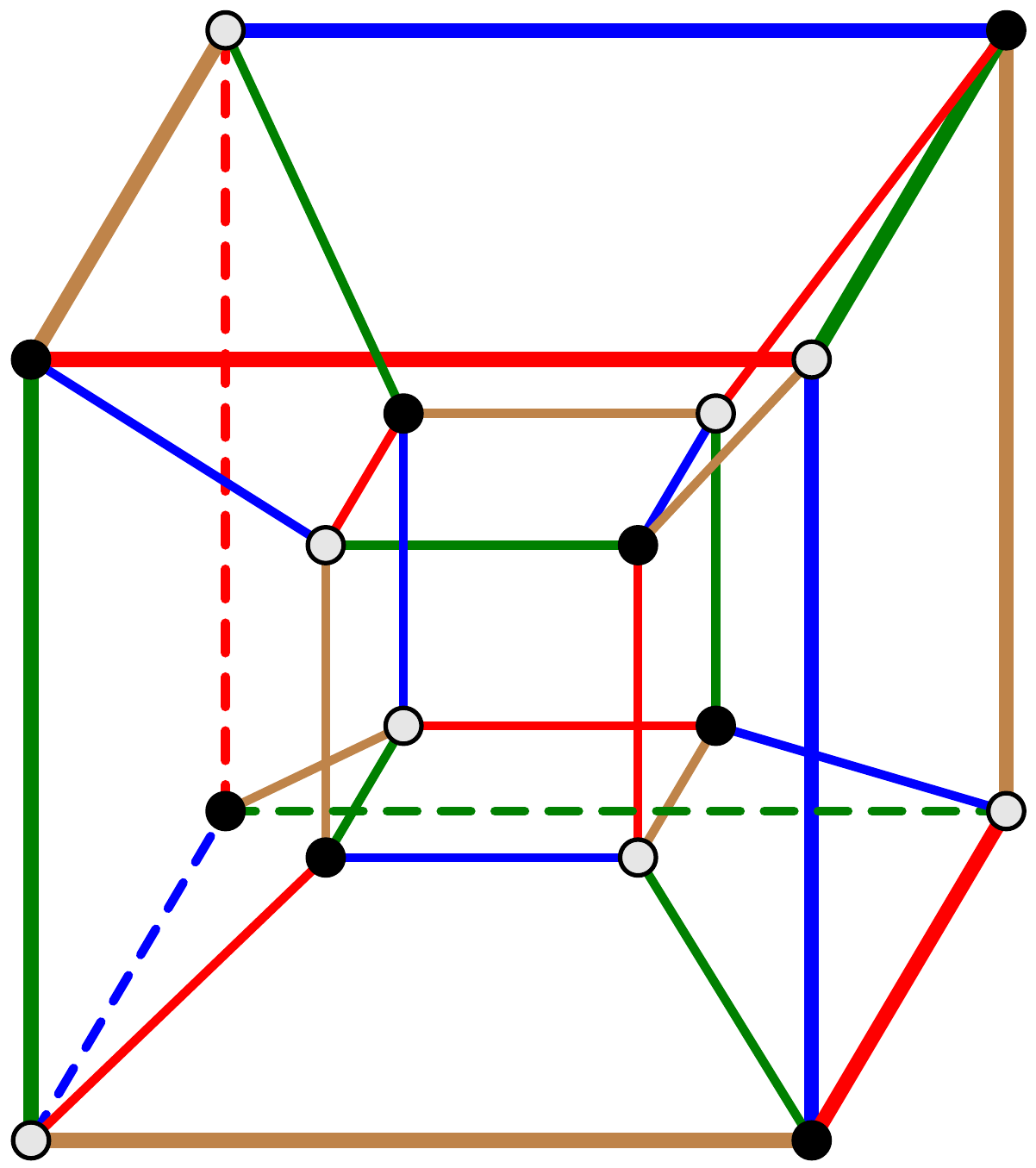}
\caption{The colourful polytope $\hat{\qo}$.}
\label{TheOne}
\end{center}
\end{figure}

Each of the vertices and edges of $\qo$ lift to two copies of them, and $K_{4,4}$ lifts into the graph $G$, the 1-skeleton of the hypercube $\{4,3,3\}$.
We label the vertices of $G$ as $\tilde{v}$ and $-\tilde{v}$, where $v$ is a vertex of $\qo$ in such a way that the sets of vertices $\{\tilde{v} \mid v \in V(\qo)\}$ and $\{-\tilde{v} \mid v \in V(\qo)\}$ are the sets of two disjoin cubes of $\{4,3,3\}$.
The 4-gons of $\qo$ lift into 8-gons of $\hat{\qo}$, implying that $\hat{\qo}$ has Schl\"afli type $\{8,3,3\}$.
Figure~\ref{2facetsinR4} shows a 2-face of $\hat{\qo}$ and its two incident facets.
Each of the 4 facets is of type $\{8,3\}$, has 16 vertices, 24 edges and 6 faces.
We note that $\hat{\qo}$ is again a colourful polytope and hence every symmetry of $\hat{\qo}$ induces a permutation of the colours of the graph $G$.
Using Figure~\ref{TheOne} it is straightforward to see that $\hat{\qo}$ is not regular.

\begin{figure}[htbp]
\begin{center}
\includegraphics[width=15cm]{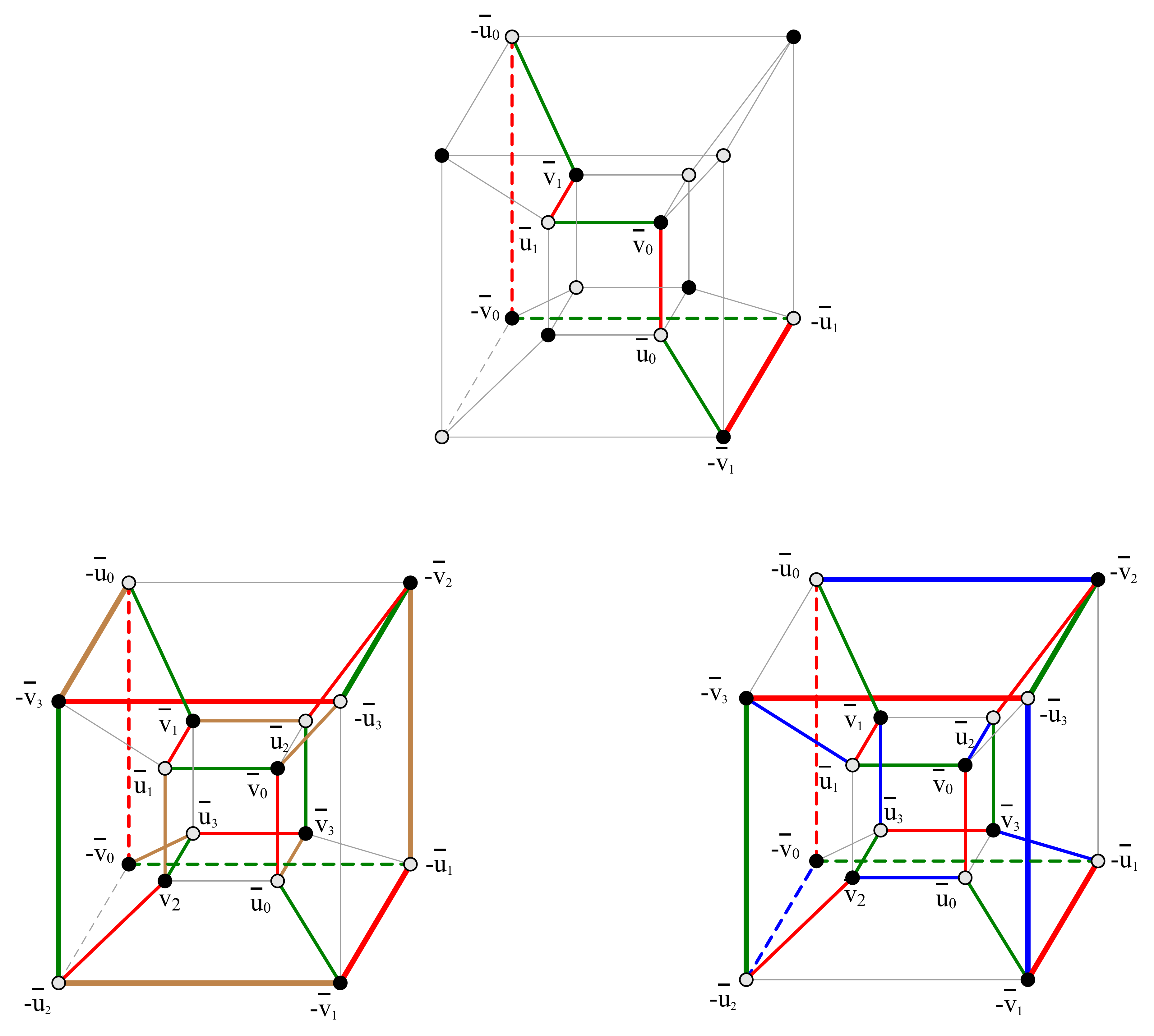}
\caption{A face of $\hat{\qo}$ and its two incident facets}
\label{2facetsinR4}
\end{center}
\end{figure}

Finally, the polytope $\hat{\qo}$ is chiral. This can be seen with arguments analogous to those in the previous section, using the symmetry group of the 4-cube $\{4,3,3\}$ instead of that of the hemi-cube. 
Furthermore, it is not difficult to see that the 2-faces of $\hat{\qo}$ are helices in $\mathbb{R}^4$ and that the stabiliser of the pair consisting of the 2-face and any of the 3-faces in Figure~\ref{2facetsinR4}, as permutation of the vertices of $\hat{\qo}$, is $$(\tilde{v_0}, \tilde{u_0}, -\tilde{v_1}, -\tilde{u_1}, -\tilde{v_0}, -\tilde{u_0}, \tilde{v_1}, \tilde{u_1})(\tilde{u_2}, \tilde{v_2},-\tilde{u_3},\tilde{v_3}, -\tilde{u_2}, -\tilde{v_2},\tilde{u_3},-\tilde{v_3}),$$
implying that the 1-step rotation in that 2-face is also the 3-step rotation in the other red-green 2-face.
Hence, the two components of the above isometry are a 1-step 8-fold rotation followed by a perpendicular 3-step 8-fold rotation.
We have established the following theorem.

\begin{theorem}
The polytope $\hat{\qo}$ is a chiral 4-polytope of full rank.
\end{theorem}



\begin{thebibliography}{99}

\bibitem{colour}
G. Araujo-Prado, I. Hubard, D. Oliveros, E. Schulte.
Colorful polytopes and graphs.
{\em Israel Journal of Mathematics.}
Volume 195, Issue 2,  2013, Pages 647--675.

\bibitem{proj1}
J.L. Arocha, J. Bracho, L. Montejano.
Regular projective polyhedra with planar faces I.
{\em Aequationes Matematicae.}
Volume 59, 2000, Pages 55--73.

\bibitem{proj2}
J. Bracho.
Regular projective polyhedra with planar faces II.
{\em Aequationes Matematicae.}
Volume 59, 2002, Pages 160--176.

\bibitem{dress1}
A.W.M. Dress
A combinatorial theory of Gr\"unbaum's new regular polyhedra, I: Gr\"unbaum;s new polyhedra and their automorphism group.
{\em Aequationes Mathematicae.}
Volume 23, 1981, Pages 252--265.

\bibitem{dress2}
A.W.M. Dress
A combinatorial theory of Gr\"unbaum's new regular polyhedra, II: complete enumeration.
{\em Aequationes Mathematicae.}
Volume 29, 1985, Pages 222--243.

\bibitem{grove} L.C.Grove and C.T.Benson, {\em Finite Reflection Groups\/} (2nd edition), Graduate Texts in Mathematics, Springer-Verlag (New York-Heidelberg-Berlin-Tokyo, 1985).

\bibitem{branko}
B. Gr\"ubaum.
Regular polyhedra - old and new.
{\em Aequationes Mathematicae.}
Volume 16, 1977, Pages 1--20.

\bibitem{peter}
P. McMullen.
Regular Polytopes of Full Rank.
{\em Discrete and Computanional Geometry.}
Volume 32, Issue 1, 2004, Pages 1--35.

\bibitem{arp}
P. McMullen, E. Schulte.
{\em Abstract Regular Polytopes}.
Cambridge University Press, 2002.

\bibitem{egon1}
E. Schulte.
Chiral polyhedra in ordinary space, I.
{\em Discrete and Computational Geometry}
Volume 32, Issue 1, 2004,  Pages 55--99.

\bibitem{egon2}
E. Schulte.
Chiral polyhedra in ordinary space, II.
{\em Discrete and Computational Geometry}
Volume 34, Issue 2, 2005, Pages 181-229.

\bibitem{chiregas}
   E. Schulte and A. I. Weiss, Chiral polytopes, {\em Applied Geometry and Discrete
Mathematics (The ``Victor Klee Festschrift'')}, DIMACS Ser. Discrete Math. Theoret.
Comput. Sci. {\bf 4} (1991), (eds. P. Gritzmann and B. Sturmfels), 493--516.


\end{thebibliography}
\end{document}